\documentclass[a4paper,11pt]{article}
\usepackage[T1]{fontenc}
\usepackage[latin1]{inputenc}

\usepackage{amsmath,amsfonts}
\usepackage{amsthm}
\usepackage{amssymb, amsopn}
\usepackage[USenglish]{babel}
\usepackage{amsfonts}
\usepackage{enumerate}
\usepackage{verbatim}
\usepackage{graphicx}
\usepackage{verbatim}
\usepackage[left=3cm, right=3cm, top=3cm, bottom=3cm]{geometry}
\usepackage{tikz}
\usepackage{csquotes}
\usepackage{mathscinet} 

\usepackage[final,                  
            pdfpagelabels,
            pdfstartview = {FitH},  
            bookmarks,              
            colorlinks,             
            plainpages = false,     
            linktoc=all,            
            linkcolor=black,        
            citecolor=blue,         
            urlcolor=blue,          
            filecolor=black         
            ]{hyperref}

\usetikzlibrary{matrix,arrows,decorations.pathmorphing,decorations.markings}

\theoremstyle{plain}
\newtheorem{teo}{Theorem}[section]
\newtheorem{lemma}[teo]{Lemma}
\newtheorem{prop}[teo]{Proposition}
\newtheorem{cor}[teo]{Corollary}
\newtheorem{question}[teo]{Question}

\theoremstyle{definition}
\newtheorem{dfn}[teo]{Definition}
\newtheorem{ex}[teo]{Example}
\newtheorem{rmk}[teo]{Remark}

\renewcommand{\bar}{\overline}
\newcommand{\C}{\mathbb{C}}
\newcommand{\Z}{\mathbb{Z}}

\newcommand{\Q}{\mathbb{Q}}

\newcommand{\ra}{\rightarrow}
\newcommand{\cu}{\subseteq}

\newcommand{\lra}{\longrightarrow}

\renewcommand{\sl}{\mathfrak{sl}}

\newcommand{\ragamma}{\stackrel{\gamma^\vee}{\longrightarrow}}
\newcommand{\ralabel}[1]{\stackrel{#1}{\longrightarrow}}
\newcommand{\kk}{\mathbb{K}}
\newcommand{\dbg}[1]{\mathfrak{B}_\Phi^{#1-\text{deg}}}
\newcommand{\BB}{\mathfrak{B}_\Phi}

\newcommand{\Address}{
  \bigskip{\footnotesize

  \textsc{Max-Planck-Institut f\"ur Mathematik, Bonn, Germany}\par\nopagebreak
  \textit{E-mail address}: \texttt{leonardo@mpim-bonn.mpg.de}
}}

\DeclareMathOperator{\Ker}{Ker}

\DeclareMathOperator{\spa}{span}

\DeclareMathOperator{\sgn}{sgn}
\DeclareMathOperator{\coeff}{coeff}
\DeclareMathOperator{\ch}{char}
\DeclareMathOperator{\hgt}{ht}

\title{The Hard Lefschetz Theorem in positive characteristic for the Flag Varieties}
\author{Leonardo Patimo}

\begin{document}

\maketitle
\begin{abstract}
For a given flag variety, we characterize the primes $p$ for which there exists a weight $\lambda$ such that the hard Lefschetz theorem holds for multiplication by $\lambda$  on the cohomology of the flag variety
with coefficients in an infinite field of characteristic $p$.
\end{abstract}

\section{Introduction}

The hard Lefschetz theorem states that, for  a smooth projective complex variety $X$ of (complex) dimension $d$, if $\lambda\in H^2(X,\Q)$ is the first Chern class of an ample line bundle then multiplication by $\lambda^{k}$ 
induces an isomorphism between $H^{d-k}(X,\Q)$ and $H^{d+k}(X,\Q)$. 

This is no longer true when we consider the integral cohomology $H^*(X,\Z)$, or the cohomology $H^*(X,\kk)=H^*(X,\Z)\otimes^L_\Z \kk$ with coefficients in an arbitrary field $\kk$ of characteristic $p>0$.

\begin{dfn}
Let $d\geq 0$ and $V=\bigoplus_{k=0}^{2d}V^k$ be a graded finite dimensional $\kk$-vector space. Let $f:V\ra V$ be a map of degree $2$ (i.e. $f(V^k)\cu V^{k+2}$ for any $k$). We say that $f$ has the \emph{Hard Lefschetz Property} on $V$ (HLP for short)
if for any $0< k\leq d$ the map $f^k:V^{d-k}\ra V^{d+k}$ is an isomorphism.

If $V$ is a graded $\kk$-algebra we say that $\eta\in V^2$ has the HLP on $V$  if the multiplication by $\eta$ has the HLP.
\end{dfn}

Let $G$ be a complex semisimple group and $B$ be a Borel subgroup of $G$.
We denote by $X=G/B$ the flag variety of $G$: it is a smooth projective complex variety.

We will answer the following:
\begin{question}\label{question} 
Let $\kk$ be an arbitrary infinite field of characteristic $p$.
For which primes $p$ does there exist $\lambda\in H^2(X,\kk)$ such that $\lambda$ has the Hard Lefschetz Property on  $H^{*}(X,\kk)$?
\end{question}

The flag variety $X$ is a fundamental object in representation theory, and many geometric properties of $X$ have direct consequences in term of representations.

A first example occurs in characteristic $0$, where the hard Lefschetz theorem is one of the main ingredients in the proof of the Kazhdan-Lusztig conjectures \cite{KL}, which give a formula for computing the characters of simple infinite dimensional highest weight modules over the Lie algebra of $G$ (the importance of the hard Lefschetz theorem is underlined in \cite{EW1}).
It seems therefore interesting and natural to look for Hodge Theoretic properties also in the positive characteristic setting.

Let now $\kk$ be an algebraically closed field of characteristic $p$ and let $G^\vee_\kk$ be the Langlands dual group of $G$, defined over $\kk$.
Lusztig's conjecture \cite{Lu} predicts a formula for the characters of irreducible $G^\vee_\kk$-modules in terms of affine Kazhdan-Lusztig polynomials.
Lusztig's conjecture was proven for $p$ very large \cite{AJS}, but the only explicit bound known is huge \cite{Fie} (roughly $p>n^{n^2}$, where $n=$rank$(G^\vee_\kk)$).
In contrast, Williamson \cite{W5} found a family of counterexamples to Lusztig's conjectures for $p=O(c^n)$, with $c\sim 1,101$ (see also \cite{W6} for a detailed survey on the subject).

It is currently still an open problem to understand more precisely where Lusztig's conjecture holds. 
A recent approach by Fiebig \cite{Fie2} shows how the local hard Lefschetz theorem on Schubert varieties in affine flag varieties implies Lusztig's conjecture.
However, this gives only a partial answer since it is still unclear how to find effective bounds for the hard Lefschetz theorem in the general setting of affine flag varieties.

Nevertheless, we can give a complete answer to Question \ref{question}, and this can be seen as a further step in this approach to understand Lusztig's conjecture.

\subsection*{Structure of the paper}

In \S 2 we discuss in details our main result, Theorem \ref{teo1}: for $p$ larger than the number of positive roots then we can answer Question \ref{question} affirmatively. 

Using basic Schubert calculus, in \S 3 we translate the original problem, which is geometric in nature, into a combinatorial one, which is expressed only in terms of the Bruhat graph.
In \S 4 we show how the Bruhat graph can be ``degenerated'' into a product of simpler graphs (corresponding to maximal parabolic subgroups), and that it is enough to show hard Lefschetz theorem for the latter.

We discuss when HLP holds for those simpler graphs (for good choices of the maximal parabolic subgroups) in \S 5. In \S 6 we discuss the HLP for artinian complete intersection monomial rings, i.e. rings of the form $\kk[x_1,x_2,\ldots,x_n]/(x_1^{d_1},x_2^{d_2},\ldots,x_n^{d_n})$, opportunely graded: this allows to make use of the knowledge of the HLP for the single factors to investigate the HLP for a product of graphs.

Finally in \S 7 we put things together to obtain a proof of Theorem \ref{teo1}.

\subsection*{Acknowledgements}
I wish to warmly thank my Ph.D. supervisor Geordie Williamson for many precious comments and discussion, and for suggesting me that Stembridge's formula could be a good starting point. 
I would also like to thank
the referee for many useful comments.

Some of this work was completed during a research stay at the RIMS in Kyoto.  
I was supported by the Max Planck Institute for Mathematics. 

\section{Main result}

Let $G$ be a complex connected semisimple group, $B\cu G$ be a Borel subgroup and $T\cu B$ be a maximal torus.
Let $W=N_G(T)/T$ be the Weyl group and $\Phi$ be the root system of $G$. Let $\Phi^+\cu \Phi$ be the set of positive roots, consisting of the roots which occur in the Lie algebra of $B$.
Let $\Phi^\vee$ denote the set of coroots and, for any $\gamma\in \Phi$, let $\gamma^\vee$ denote the corresponding coroot.
Let $X=G/B$ be the flag variety of $G$. It is a smooth projective variety of complex dimension $d:=|\Phi^+|=\frac12 |\Phi|$.

We recall a few facts about the cohomology of $X$, see \cite[\S1]{BGG} for more details.
For $w\in W$, let $X_w=\bar{BwB/B}\cu X$ be the Schubert variety corresponding to $w$.
The elements $[X_w]\in H_{2\ell(w)}(X,\Z)$, the fundamental classes of the Schubert varieties $X_w$, are a basis of the homology of $X$. 
By taking the dual basis we obtain a basis $\{P_w\}_{w\in W}$, with $P_w\in H^{2\ell(w)}(X,\Z)$, of the integral cohomology of $X$. We call $\{P_w\}_{w\in W}$ the \emph{Schubert basis}.

Because $H^*(X,\Z)$ is free, for any field $\kk$ we have $H^*(X,\kk)\cong H^*(X,\Z)\otimes \kk$. Therefore the Schubert basis induces a basis $\{P_w\otimes 1\}_{w\in W}$ of $H^*(X,\kk)$. We will denote $P_w\otimes 1$ simply by $P_w$.

If moreover the group $G$ is simply connected then $H^2(X,\Z)$ can be identified with the character lattice, i.e.
$$H^2(X,\Z)\cong X^*(T)=\{\lambda:T\ra \C^*\mid \lambda\text{ hom. of algebraic groups}\}.$$

We denote by $\langle-,-\rangle$ the pairing between weights and coroots such that for any reflection $s_\alpha$ we have $s_\alpha(\lambda)=\lambda-\langle\lambda,\alpha^\vee\rangle\alpha$.
Then $\langle-,-\rangle$ can be extended to a $\kk$-valued pairing between $H^2(X,\kk)$ and $\Z\Phi^\vee$. We will abuse notation and refer to the elements of $H^2(X,\kk)$ as weights.

A first partial answer to the Question \ref{question} was given by Stembridge. In \cite{Ste} he computes explicitly the map $\lambda^d:H^0(X,\Z)\ra H^{2d}(X,\Z)$. We have:
\begin{equation}\label{stemfor}
\lambda^d\cdot P_e=|\Phi^+|!\prod_{\alpha\in\Phi^+}\frac{\langle \lambda,\alpha^\vee\rangle}{\hgt(\alpha)}P_{w_0}
\end{equation}
where $e\in W$ is the identity and $w_0\in W$ 
is the longest element of $W$.
The height of a root, here denoted by $\hgt(\alpha)$, is the sum of its coordinates when expressed in the basis of simple roots.

From Stembridge's formula \eqref{stemfor} it follows that if $\kk$ is a field of characteristic $p$ and $p$ does not divide $|\Phi^+|!$, i.e. if  $p>|\Phi^+|$, then there exists $\lambda\in H^2(X,\kk)$  such that $\lambda^d:H^0(X,\kk)\ra H^{2d}(X,\kk)$ is an isomorphism: we can take, for example,
$\rho=\frac12\sum_{\beta\in \Phi^+}\beta$ so that $\langle\rho,\alpha^\vee\rangle=1$ for every simple root $\alpha$.

\begin{rmk}\label{pleqPhi}
Let $k_i$ be the number of positive roots of height $i$. Then we have  $k_1\geq k_2\geq \ldots$ and $\sum k_i=|\Phi^+|$ (see \cite[\S 3.20]{Hum}). We can then regard $k_1\geq k_2\geq\ldots$ as a partition of $|\Phi^+|$ and consider the dual partition $m_1\geq m_2\geq\ldots$, i.e. $m_i=\#\{j\mid k_j\geq i\}$. The integers $m_i$ are the exponents of the group $W$. The values of the exponents increased by $1$ can be found in \cite[Table 1, \S 3.7]{Hum}.
We have
$$\prod_{\alpha\in \Phi^+} \hgt(\alpha)=\prod_{j\geq 1}j^{k_j}=\prod_{i\geq 1} m_i!.$$
It follows that the number
$$\frac{|\Phi^+|!}{\prod_{\alpha\in\Phi^+}\hgt(\alpha)}=\binom{|\Phi^+|}{m_1,m_2,\ldots}$$
is an integer and it is divided by $\displaystyle \binom{m_{j_1}+m_{j_2}+\ldots m_{j_r}}{m_{j_1},m_{j_2},\ldots, m_{j_r}}$, for any finite subset $\{j_1,j_2,\ldots,j_r\}$ of $\mathbb{N}$.

Therefore, from Stembridge's formula, it also follows that there cannot exist $\lambda$ such that $\lambda^d$ is an isomorphism if \begin{equation}\label{divides}
p |\frac{|\Phi^+|!}{\prod_{\alpha\in\Phi^+}\hgt(\alpha)}.
\end{equation}
Now, using the known explicit values of the exponents, one can check that, if $p$ is a prime such that $p\leq |\Phi^+|$ and $p$ is not as in Table \ref{table}, then \eqref{divides} holds.

\begin{table}
\caption{}\label{table}
\begin{center}

\begin{tabular}{ |c | c | c | c |}
\hline
$\Phi$ & $p$ & $|\Phi^+|$ & $\exists \lambda$ with HLP? \\
\hline
$A_2$ & $2$ & $3$ & Yes\\
$B_2$ & $3$ & $4$ & Yes\\
$G_2$ & $5$ & $6$ & Yes\\
$B_3$ & $5$ & $9$ & No\\
$C_3$ & $5$ & $9$ & No\\
$F_4$ & $5$ & $24$ & No\\

\hline
\end{tabular}
\end{center}
\end{table}

However, in type $B_3$, $C_3$ and $F_4$ we can compute explicitly, with the help of the software Magma \cite{Magma}, the map $\lambda^{d-2}:H^2(X,\kk)\ra H^{2d-2}(X,\kk)$ for an arbitrary weight $\lambda$. 
We get that $\lambda^{d-2}$ is never an isomorphism if $\kk$ is a field of characteristic $5$. 
\end{rmk}

Here we give a complete answer to Question \ref{question}. The main result is the following:

\begin{teo}\label{teo1}
Let $\kk$ be an infinite field of characteristic $p>0$. Then there exists
 $\lambda\in H^2(X,\kk)$ such that the hard Lefschetz theorem holds for $\lambda$ on $H^*(X,\kk)$ if and only if $p>|\Phi^+|$ or $\Phi$ and $p$ are as in the first three lines of Table \ref{table}.
\end{teo}

\begin{rmk}\label{rank2}
Let $X$ be of type $B_2$ and $\kk$ be an infinite field of characteristic $3$. 
We label the simple roots in the Dynkin diagram as $\alpha\Rightarrow \beta$.
 Let $\lambda=a\varpi_\alpha+b\varpi_\beta$ be an arbitrary weight, where $a,b\in \kk$ and $\varpi_\alpha,\varpi_\beta$ are the fundamental weights. We can compute explicitly the hard Lefschetz determinants:
 \begin{itemize}
  \item $D_4(a,b):=\det(\lambda^4:H^0(X,\mathbb{F}_5)\ra H^8(X,\mathbb{F}_5))=4ab(a+b)(a+2b)$;
  \item $D_2(a,b):=\det(\lambda^2:H^2(X,\mathbb{F}_5)\ra H^6(X,\mathbb{F}_5))=-(a^2+2ab+2b^2)$.
 \end{itemize}
The polynomials $D_2$ and $D_4$ are not identically zero, so there exists $\lambda$ with the HLP. For instance, we can choose $\lambda=a\varpi_\alpha+\varpi_\beta$, with $a\in \kk\setminus\{0,1,2\}$ such that it is not a root of the polynomial $x^2+2x+2$.
 
Similar elementary computations show that there exists $\lambda$ with the HLP on $H^*(X,\kk)$ if $X$ is of type $A_2$ (resp. $G_2$) and $\kk$ is a infinite field of characteristic $2$ (resp. $5$).
Thus $A_2$, $B_2$ and $G_2$ are the only types for which there exists $\lambda\in H^2(X,\kk)$ 
with the HLP for a field $\kk$ such that char$(\kk)\leq |\Phi^+|$.
\end{rmk}

\begin{rmk}\label{b2}
 The situation is more subtle if one considers the case of a finite field. 

For example, let $X$ be of type $B_2$ and let $\kk=\mathbb{F}_5$.
Similarly to Remark \ref{rank2}, let $\lambda=a\varpi_\alpha+b\varpi_\beta$ with $a,b\in \mathbb{F}_5$. Notice that in this case we have  $D_2(a,b)=-(a^2+2ab+2b^2)=-(a+3b)(a+4b)$.
It follows that there are no $a,b\in \mathbb{F}_5$ such that $\lambda$ has the HLP on $H^*(X,\mathbb{F}_5)$, although $5>|\Phi^+|=4$.
 \end{rmk}

\section{The Bruhat Graph of a Root System}

\begin{dfn}
We define the \emph{Bruhat graph} $\BB$ of $\Phi$. The vertices of the graph are the element of $W$.
There is an arrow $w\ragamma v$ for $v,w\in W$ if $\ell(v)=\ell(w)+1$ and $wt_\gamma=v$, where $t_\gamma$ is the reflection corresponding to the positive root $\gamma$.
\end{dfn}

\begin{rmk}
In this paper our terminology Bruhat graph is somewhat non-standard. For example, in \cite[Definition 1.1]{Dye} it is defined to be the graph whose vertices are the elements of $W$, 
in which there is an arrow $w\ragamma v$ for $v,w\in W$ whenever $v=wt_\gamma$ and $\ell(v)>\ell(w)$.
\end{rmk}

\begin{ex}\label{Graph}
If $G=SL_3(\C)$, then $\Phi$ is the root system of type $A_2$ and $W\cong \mathcal{S}_3$, the symmetric group on $3$ elements. 
It is generated by the simple transpositions $s$ and $t$. Let $\alpha$ and $\beta$ be the two simple coroots corresponding to $s$ and $t$. The Bruhat graph $\BB$ is:

\begin{center}
 \begin{tikzpicture}
 	\node[shape=circle,draw=black,minimum height=8mm,minimum width=8mm] (e) at (0,0) {$e$};
    \node[shape=circle,draw=black,minimum height=8mm,minimum width=8mm] (s) at (-1,1) {$s$};
    \node[shape=circle,draw=black,minimum height=8mm,minimum width=8mm] (t) at (1,1) {$t$};
    \node[shape=circle,draw=black,minimum height=8mm,minimum width=8mm] (st) at (1,2.5) {$st$};
    \node[shape=circle,draw=black,minimum height=8mm,minimum width=8mm] (ts) at (-1,2.5) {$ts$};
    \node[shape=circle,draw=black,minimum height=8mm,minimum width=8mm] (sts) at (0,3.5) {$sts$};
    
    \path [->,color=red, pos=0.7] (e) edge node[scale=0.8,below] {$\alpha$} (s);
    \path [->,color=blue, pos=0.7] (e) edge node[scale=0.8,below] {$\beta$} (t);
    \path [->,color=blue, pos=0.7] (s) edge node[scale=0.8,above] {$\beta$} (st);
    \path [->,color=violet, pos=0.5] (s) edge node[scale=0.8,left] {$\alpha+\beta$} (ts);
    \path [->,color=violet, pos=0.5] (t) edge node[scale=0.8,right] {$\alpha+\beta$} (st);
  	\path [->,color=red, pos=0.7] (t) edge node[scale=0.8,above] {$\alpha$} (ts);
    \path [->,color=red, pos=0.3] (st) edge node[scale=0.8,above] {$\alpha$} (sts);
   	\path [->,color=blue, pos=0.3] (ts) edge node[scale=0.8,above] {$\beta$} (sts);
 \end{tikzpicture}
\end{center}
\end{ex}

We recall Pieri-Chevalley's formula \cite[Theorem 3.14]{BGG}. Let $\lambda\in H^2(X,\Z)$ be a weight. Then 
$$\lambda\cdot P_w=\sum_{w\ragamma {v}}\langle \lambda, \gamma^\vee\rangle P_{v}.$$

If $\ell(v)-\ell(w)=k$, let $C_{w,v}(\lambda)\in \Z$ be defined by
$$\lambda^k\cdot P_w=\sum_{\ell(v)=\ell(w)+k}C_{w,v}(\lambda)P_v.$$
Then we have
$$C_{w,v}(\lambda)=\sum \langle \lambda, \gamma_1^\vee\rangle \langle \lambda, \gamma_2^\vee\rangle\ldots \langle \lambda,\gamma_k^\vee\rangle=\sum C_{w,w_1}(\lambda)C_{w_1,w_2}(\lambda)\ldots C_{w_{k-1},v}(\lambda)$$
where the sum runs over all paths $w\ralabel{\gamma_1^\vee}w_1\ralabel{\gamma_2^\vee}w_2\ralabel{\gamma_3^\vee} \ldots \ralabel{\gamma_k^\vee}v$ connecting $w$ to $v$.

Let $S\cu W$ be the set of simple reflections and $I\cu S$ be a subset. Let $W_I$ be the subgroup generated by the simple reflections in $I$. We denote by $W^I\cu W$ the set of representatives of minimal length in $W/W_I$.
If $J\cu I$, then $W_I^J$ is well defined. Let $\Phi(I)$ be the sub-root system of $\Phi$ generated by the simple roots in $I$. 

If $w\in W$, there exist unique $w'\in W^I$ and $w''\in W_I$ such that $w=w'w''$.  Moreover if $w=w'w''$ with $w'\in W^I$ and $w''\in W_I$ we have $\ell(w)=\ell(w')+\ell(w'')$. 

We denote by $\geq$ the Bruhat order in $W$.

\begin{lemma}\label{poset}
Let $w\geq v$ in $W$. Then $w'\geq v'$.
\end{lemma}
\begin{proof}
The projection $W\ra W/W_I$ is a morphism of posets. This follows from \cite[Lemma 2.2]{Dou}.
\end{proof}

Let $P_I$ be the parabolic subgroup $B\cu P_I\cu G$ corresponding to the subset $I$. The projection $G/B\ra G/P_I$ induces an injective map $H^*(G/P_I,\Z)\ra H^*(G/B,\Z)$. The image is the subspace generated by all the $P_w$, with $w\in W^I$, cf. \cite[Theorem 5.5]{BGG}. 

For $s\in S$, let $\alpha_s\in \Phi$ denote the corresponding simple root and  $\varpi_s$ denote the corresponding fundamental weight, i.e. $\langle \varpi_s,\alpha_t^\vee\rangle =\delta_{s,t}$ for any $t\in S$.

From Pieri-Chevalley's formula we get $\varpi_s=P_s$ for any $s\in S$.
So the subspace $H^2(G/P_I,\Z)\cu H^2(G/B,\Z)$ has as a basis the set $\{\varpi_s\}_{s\in S\setminus I}$. 

\begin{dfn}
Let $I$ be a subset of $S$. We define the \emph{degeneration map} $\pi_I:\Z \Phi^\vee\ra \Z\Phi^\vee$ as follows: 

$$ \pi_I\left(\sum_{s\in S}c_s\alpha_s^\vee\right)=\begin{cases}\sum_{s\in S}c_s\alpha_s^\vee=\sum_{s\in I}c_s\alpha_s^\vee&\text{if }c_s=0\; \forall s\in S\setminus I\\
\sum_{s\in S\setminus I}c_s\alpha_s^\vee & \text{otherwise.}\end{cases}$$ 
\end{dfn}

\begin{dfn}\label{pbg}
Let $I\cu S$ be a subset. 
The \emph{parabolic Bruhat graph} $\BB^I$ is a graph whose vertices are the elements in $W^I$. For any edge $w\ragamma  v$ in $\BB$, with $w,v\in W^I$, we put an edge $w\xrightarrow{\pi_I(\gamma^\vee)}v$ in $\BB^I$, where $\pi_I:\Z\Phi^\vee \ra \Z\Phi^\vee$
is the degeneration map. 
\end{dfn}

Notice that if $w,v\in W^I$ with $w\ragamma v$, then $\gamma\not\in \Phi(I)$. Hence in Definition \ref{pbg} only the second case in the definition of the degeneration map $\pi_I$ is used.

We see easily from Pieri-Chevalley's formula that the graph $\BB^I$ describes the multiplication by $\lambda\in H^2(G/P_I,\Z)$ in $H^*(G/P_I,\Z)$ in the Schubert basis $\{P_w\}_{w\in W^I}$, i.e.
$$\lambda\cdot P_w =\sum_{w\ralabel{\delta} v\in\BB^I}\langle \lambda,\delta\rangle P_{v}.$$

\section{The Degeneration of a Bruhat Graph}

Fix now $\kk$ an arbitrary infinite field and let $\lambda\in H^2(X,\kk)$ be an arbitrary weight. 

We label the elements of $S=\{1,2,\ldots,n\}$, then we can express $\lambda$ as $\sum_{i=1}^n x_i\varpi_i$ with $x_i\in \kk$. From now on we will regard the $x_i$ as indeterminate variables.

After we fix arbitrarily an ordering of the Schubert basis (or, equivalently, of the elements of $W$) the map $\lambda^k:H^{d-k}(X,\kk)\ra H^{d+k}(X,\kk)$ can be thought of as a square matrix with number of columns equal to the number of elements of length $(d-k)/2$ in $W$. 
Taking the determinant we obtain a polynomial $D_k(\lambda)=D_k(x_1,\ldots,x_n)$. Since the field $\kk$ is infinite, the existence of $\lambda$ satisfying the Lefschetz property is equivalent to  $D_k(x_1,\ldots,x_n)\neq 0$,
for all  $0< k\leq n$.

The polynomials $D_k(\lambda)$ appear to be hard to compute explicitly. However, it is sufficient for our purposes to compute a single term in $D_k(\lambda)$:
its leading term in the lexicographic order $x_1>x_2>\ldots>x_n$.

\begin{dfn}
Let $I$ be a subset of $S$. We say that $w$ \textit{$I$-dominates} $v$ if $w=w'w''$, $v=v'v''$, with $w',v'\in W^I$, $w'',v''\in W_I$ and $w'\geq v'$, $w''\geq v''$ ($\geq$ is the usual Bruhat order).
We say that an edge $w\ragamma v$ is $I$-\textit{relevant} if $v$ $I$-dominates $w$. A path connecting $w$ to $v$ is $I$-relevant if all its edges are $I$-relevant.
\end{dfn}
In view of Lemma \ref{poset} we have that $v$ $I$-dominates $w$ if and only if $v\geq w$ and $v''\geq w''$.

\begin{lemma}\label{deodhar}
 Let $v,w\in W$ such that $v'=w'$. Then $v\geq w$ if and only if $v''\geq w''$.
\end{lemma}
\begin{proof}
Let $s\in S$ such that $sv'<v'$. We have $sv'\in W^I$ \cite[Lemma 3.1]{Deo}. Moreover, by   \cite[Theorem 1.1]{Deo}, we have $v\geq w$ if and only if $sv\geq sw$, so we can easily conclude by induction on $\ell(v')$.
\end{proof}

\begin{lemma}\label{crucial}
Let $w\ragamma v$ be an edge in $\BB$. Then $w\ragamma  v$ is $I$-relevant if and only if $\ell(v')\leq \ell(w')+1$.
\end{lemma}
\begin{proof}
If $w\ragamma v$ is $I$-relevant, then $\ell(v)=\ell(w)+1$ and $\ell(v'')\geq \ell(w'')$, so clearly $\ell(v')\leq \ell(w')+1$.

Conversely, if $\ell(v')=\ell(w')$ then $v'=w'$ because of Lemma \ref{poset}. Therefore $v''>w''$ by Lemma \ref{deodhar} and $w\ragamma v$ must be $I$-relevant.

It remains to consider the case $\ell(v')=\ell(w')+1$, or equivalently $\ell(v'')=\ell(w'')$. We claim that in this case we have $v''=w''$, whence in particular $w\ragamma v$ is $I$-relevant.
The claim is proven by induction on $\ell(v'')=\ell(w'')$. The case $\ell(v'')=0$ is clear.
Let $s\in I$ such that $v''s<v''$. This implies, again by \cite[Theorem 1.1]{Deo}, that $w\leq vs$ or $ws\leq vs$.
If $w\leq vs< v$, then $w=vs$. Thus we have $w'=(vs)'=v'$, which is a contradiction since $\ell(v')=\ell(w')+1$.
If $ws\leq vs$ then $ws\xrightarrow{s(\gamma)^\vee}vs$ is an edge in $\BB$. Since $v'=(vs)'$ and $w'=(ws)'$ we have $\ell((vs)')=\ell((ws)')+1$ and $\ell((ws)'')=\ell((vs)'')=\ell(v'')-1$.
Hence we can apply the inductive hypothesis to get $w''s=v''s$, thus $w''=v''$.
\end{proof}

In other words, the proof of Lemma \ref{crucial} shows that an edge $w\ragamma v$ in $\BB$ is $I$-relevant if and only if $v'=w'$ or $v''=w''$.

\begin{dfn}
The $I$-\emph{degenerate Bruhat Graph} $\dbg{I}$ is a graph having the same vertices as the Bruhat graph $\BB$. The edges in $\dbg{I}$ are the $I$-relevant edges in $\BB$: for any $I$-relevant edge $w\ragamma v$ in $\BB$ 
we put an edge $w\xrightarrow{\pi_I(\gamma^\vee)}v$ in $\dbg{I}$.
\end{dfn}

In particular, in the case $I=S\setminus\{s\}$ the edges in $\dbg{I}$ are all labeled by $m\alpha_s$, with $m\in \mathbb{N}_{>0}$,  or by a root in $\Phi(I)$.

\begin{ex}
Let $\Phi$ be the root system of type $A_2$ as in the Example \ref{Graph} and let $I=\{t\}$. 
Then $ts$ does not $I$-dominate $t$, although $ts>t$ in the Bruhat order. In fact $(ts)''=e\not>t=t''$. Thus the edge $t\lra ts$ is not $\{t\}$-relevant. The degenerate Bruhat graph $\dbg{\{t\}}$ is:

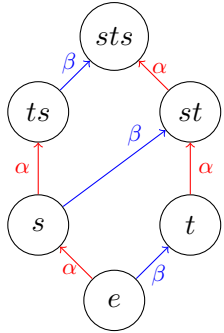
\begin{center}
 \begin{tikzpicture}
 	\node[shape=circle,draw=black,minimum height=8mm,minimum width=8mm] (e) at (0,0) {$e$};
    \node[shape=circle,draw=black,minimum height=8mm,minimum width=8mm] (s) at (-1,1) {$s$};
    \node[shape=circle,draw=black,minimum height=8mm,minimum width=8mm] (t) at (1,1) {$t$};
    \node[shape=circle,draw=black,minimum height=8mm,minimum width=8mm] (st) at (1,2.5) {$st$};
    \node[shape=circle,draw=black,minimum height=8mm,minimum width=8mm] (ts) at (-1,2.5) {$ts$};
    \node[shape=circle,draw=black,minimum height=8mm,minimum width=8mm] (sts) at (0,3.5) {$sts$};
    
    \path [->,color=red, pos=0.7] (e) edge node[scale=0.8,below] {$\alpha$} (s);
    \path [->,color=blue, pos=0.7] (e) edge node[scale=0.8,below] {$\beta$} (t);
    \path [->,color=blue, pos=0.7] (s) edge node[scale=0.8,above] {$\beta$} (st);
    \path [->,color=red, pos=0.5] (s) edge node[scale=0.8,left] {$\alpha$} (ts);
    \path [->,color=red, pos=0.5] (t) edge node[scale=0.8,right] {$\alpha$} (st);
    \path [->,color=red, pos=0.3] (st) edge node[scale=0.8,above] {$\alpha$} (sts);
   	\path [->,color=blue, pos=0.3] (ts) edge node[scale=0.8,above] {$\beta$} (sts);
 \end{tikzpicture}
\end{center}
\end{ex}

The graph $\dbg{I}$ describes a new action $\stackrel{I}{\cdot}$ of $\lambda$ on $H^*(X,\kk)$. We say
$$\lambda\stackrel{I}{\cdot} P_w =\sum_{w\ralabel{\delta} v\in\dbg{I}}\langle \lambda,\delta\rangle P_{v}$$
where the sum runs over all edges $w\ralabel{\delta} v$ starting in $w$ in $\dbg{I}$ (or equivalently all $I$-relevant edges starting in $w$ in $\BB$).
We call it the $I$\emph{-degenerate action} of $\lambda$.

The new graph $\dbg{I}$ can be obtained as product of two smaller graphs. In fact, we have $\dbg{I}\cong\BB^I\times \mathfrak{B}_{\Phi(I)}$:
at the level of vertices we have a bijection $W=W^I\times W_I$ and, because of Lemma \ref{crucial}, for any $I$-relevant edge $w\ragamma v$ we have two cases:
\begin{itemize}
 \item $w'=v'$ and $w''t_\gamma=v''$, so $w\xrightarrow{\pi_I(\gamma^\vee)} v$ comes from the edge $w''\ragamma v''$ in $\mathfrak{B}_{\Phi(I)}$;
 \item $w''=v''$ and $w't_{w''(\gamma)}=v'$, so $w\xrightarrow{\pi_I(\gamma^\vee)} v$ comes from the edge $w'\xrightarrow{\pi_I(w''(\gamma)^\vee)}v'$ in $\BB^I$.
\end{itemize}

\begin{rmk}
It is not hard to see that the $I$-degenerate action described by $\dbg{I}$ coincides with the action on $H^*(G/P_I\times P_I/B,\kk)\cong H^*(G/P_I,\kk)\otimes H^*(P_I/B,\kk)$ defined as follows:
 if $\lambda=\sum_{i\in S} x_i\varpi_i$, $P_1\in H^*(G/P_I,\kk)$ and 
  $P_2\in H^*(P_I/B,\kk)$  then   
 $$\lambda\stackrel{I}{\cdot} (P_1\otimes P_2)=\Big(\sum_{i\in S\setminus I}x_i\varpi_i\Big)\cdot P_1\otimes P_2+P_1\otimes  \Big(\sum_{i\in I}x_i\varpi_i\Big)\cdot P_2.$$
\end{rmk}

For a polynomial $f\in \kk[x_1,\ldots,x_n]$ we denote by $\deg_i(f)$ its degree in the variable $x_i$ and by $\coeff_{i,a}(f)$ the coefficient of $x_i^a$ in $f$ (thus $\coeff_{i,a}(f)$ is an element of $\kk[x_1,\ldots x_{i-1},x_{i+1},\ldots x_n]$).
We set $\deg_i(0)=-1$.

Recall that the elements of $S$ are labeled as $\{1,2,\ldots,n\}$ and that $\lambda=\sum_i x_i\varpi_i$ is a formal linear combination of the fundamental weights. We set $I=S\setminus\{1\}$. 

We have $\deg_1(\langle\lambda,\gamma^\vee\rangle)=1$ if $\gamma\in \Phi\setminus \Phi(I)$ and $\deg_1(\langle\lambda,\gamma^\vee\rangle)=0$ if $\gamma\in \Phi(I)$. Notice that $\gamma\in \Phi(I)$ if and only if $t_\gamma\in W_I$.

\begin{lemma}\label{relevant}
Let $w,v\in W$ with $\ell(v)>\ell(w)$. Then:
\begin{enumerate}[i)]
\item \label{rel1} $\deg_1(C_{w,v}(\lambda))\leq \ell(v')-\ell(w')$ and we have equality if and only if there exists an $I$-relevant path connecting $w$ to $v$; 
\item \label{rel2}$\displaystyle \coeff_{1,\ell(v')-\ell(w')}(C_{w,v}(\lambda))\cdot x_1^{\ell(v')-\ell(w')}=
\hspace{-6pt}\sum_{\text{relevant}} \langle \lambda,\pi_I(\gamma_1^\vee)\rangle \langle \lambda,\pi_I(\gamma_2^\vee)\rangle\ldots\langle \lambda,\pi_I(\gamma_{k}^\vee)\rangle$,\\
where the sum runs over all the $I$-relevant paths $w\ralabel{\gamma_1^\vee}w_1\ralabel{\gamma_2^\vee}w_2\ralabel{\gamma_3^\vee}\ldots\ralabel{\gamma_k^\vee}v$ connecting $w$ to $v$ in $\BB$.

\end{enumerate}
\end{lemma}
\begin{proof}
\emph{i)} We start with the case $\ell(v)=\ell(w)+1$. If there are no edges connecting $w$ to $v$ in $\BB$ then there is nothing to show.

Assume that there is an edge $w\ragamma v$ in $\BB$, so that $C_{w,v}(\lambda)=\langle \lambda,\gamma^\vee\rangle$. 
If $w\ragamma v$ is not $I$-relevant by Lemma \ref{crucial} we have $\ell(v')-\ell(w')\geq 2$, and the statement follows since $\deg_1(C_{w,v}(\lambda))\leq 1$.

Assume now that $w\ragamma v$ is $I$-relevant, then $w'=v'$ or $w''=v''$. 
Since $w'w''t_\gamma=v'v''$ we see that $w'=v'$ if and only if $t_\gamma\in W_I$, i.e. if and only if $\deg_1(C_{w,v}(\lambda))=0$.

The general case $\ell(v)>\ell(w)+1$ follows since 
$$C_{w,v}(\lambda)=\sum C_{w,w_1}(\lambda)C_{w_1,w_2}(\lambda)\ldots C_{w_{k-1},v}(\lambda)$$
where the sum runs over all paths $w\lra w_1\lra w_2\lra\ldots\lra v$ in $\BB$.

\emph{ii)} We start with the case $\ell(v)=\ell(w)+1$. If there are no $I$-relevant edges in $\BB$ between $w$ and $v$ then both sides are $0$. 
If there is an $I$-relevant edge $w\ragamma v$, then $C_{w,v}(\lambda)=\langle \lambda,\gamma^\vee\rangle$ and 
$$\displaystyle \coeff_{1,\ell(v')-\ell(w')}(\langle \lambda,\gamma^\vee\rangle)\cdot x_1^{\ell(v')-\ell(w')}=\langle \lambda,\pi_I(\gamma^\vee)\rangle.$$
The general case $\ell(v)>\ell(w)+1$ easily follows.
\end{proof}

We fix now an arbitrary $k\in \{1,2,\ldots,d\}$. Let $D_k^{(1)}(\lambda)$ be the hard Lefschetz determinant of the $I$-degenerate action of $\lambda$ on $H^{d-k}(X,\kk)$, described by $\dbg{I}$, computed in the same basis used for $D_k(\lambda)$.
In other words $D_k^{(1)}(\lambda)$ is the determinant of the map $\lambda^k:H^{d-k}(G/P_I\times P_I/B)\ra H^{d+k}  (G/P_I\times P_I/B)$ described above.

\begin{lemma}\label{09}Let $\displaystyle M_k=\sum_{\ell(v)={(d+k)/2}}l(v')-\sum_{\ell(w)={(d-k)/2}}l(w')$. Then:
\begin{enumerate}[i)]\item $\deg_1(D_k(\lambda))\leq M_k $;
\item The polynomial $D_k^{(1)}(\lambda)$ is homogeneous of degree $M_k$ in $x_1$;
\item $\coeff_{1,M_k}(D_k(\lambda))\cdot x_1^{M_k}=D_k^{(1)}(\lambda)$.
\end{enumerate}
\end{lemma}
\begin{proof}
The determinant polynomial can be expressed as 
$$D_k(\lambda)=\sum_{\sigma}\sgn(\sigma)C_{w_1,\sigma(w_1)}(\lambda)C_{w_2,\sigma(w_2)}(\lambda)\ldots C_{w_{n(k)},\sigma(w_{n(k)})}(\lambda)$$
where $\sigma$ runs over all possible bijections between elements in $W$ of length $(d-k)/2$ and $(d+k)/2$  (and the sign is determined by the chosen order of the Schubert basis). Then \emph{(i)} follows from Lemma \ref{relevant}. 

The terms in the sum which contribute to $\coeff_{1,M_k}(D_k(\lambda))$ are precisely the ones coming from $I$-relevant paths, i.e. the one which are also in $D^{(1)}_k(\lambda)$, so \emph{(ii)} and \emph{(iii)} follow.
\end{proof}

We can now reiterate this procedure.
Let $S=I_0\supset I_1\supset I_2\supset\ldots \supset I_n=\emptyset$ with $I_{j-1}\setminus I_{j}=\{j\}$ for any $1\leq j \leq n$. We have a length preserving bijection of sets:
$$\Psi:W\cong W^{I_1}\times W_{I_1}^{I_2}\times \ldots \times W_{I_{n-1}}.$$
We write $\Psi(w)=\left(w^{(1)},w^{(2)},\ldots,w^{(n)}\right)$.
The degenerated graph $\BB^{(1)}:=\dbg{I_1}$ is isomorphic to $\mathfrak{B}_{\Phi}^{I_1}\times\mathfrak{B}_{\Phi(I_1)}$. It can be degenerated again into  $\BB^{(2)}:=\mathfrak{B}_{\Phi}^{I_1}\times\mathfrak{B}_{\Phi(I_1)}^{I_2\text{-deg}}\cong 
\mathfrak{B}_{\Phi}^{I_1}\times\mathfrak{B}_{\Phi(I_1)}^{I_2}\times \mathfrak{B}_{\Phi(I_2)}$, and so on up to $\BB^{(n-1)}:=\mathfrak{B}_{\Phi}^{I_1}\times\mathfrak{B}_{\Phi(I_1)}^{I_2}\times \ldots \times\mathfrak{B}_{\Phi(I_{n-1})}$.
We set $\BB^{(0)}:=\BB$ and $\BB^{(n)}:=\BB^{(n-1)}$.

\begin{dfn}
Each of the $\BB^{(j)}$ describes a new action of $\lambda$ on $H^*(X,\kk)$, which we call the \emph{$j^\text{th}$-degenerate action} and we denote by $\stackrel{j}{\cdot}$. We say that $v$ \textit{$j$-dominates} $w$ if $v^{(i)}\geq w^{(i)}$ for any $i\leq j$ and $v^{(j+1)}\ldots v^{(n)}\geq w^{(j+1)}\ldots w^{(n)}$.

We say that an edge $w\ragamma v$ is $j$-\textit{relevant} if $v$ $j$-dominates $w$. A path connecting $w$ to $v$ is $j$-relevant if all its edges are $j$-relevant.
\end{dfn}

For $1\leq j\leq n$, let $C_{w,v}^{(j)}(\lambda)$  be the coefficient of $P_v$ in $\lambda^h\stackrel{j}{\cdot}P_w$, where $\ell(v)-\ell(w)=h$. 
Thus Lemma \ref{relevant}.\ref{rel2} can be restated as: 
$$\coeff_{1,\ell(v^{(1)})-\ell(w^{(1)})}(C^{(0)}_{w,v}(\lambda))\cdot x_1^{\ell(v^{(1)})-\ell(w^{(1)})}=
C_{w,v}^{(1)}(\lambda).$$
We also have:
\begin{lemma}\label{Dkiter} Let $w,v\in W$ with $\ell(v)>\ell(w)$ and $0\leq j\leq n-1$. Then:
\begin{enumerate}[i)]
\item $\displaystyle \deg_{j+1}C^{(j)}_{w,v}(\lambda)\leq \ell(v^{(j+1)})-\ell(w^{(j+1)})$ and the equality holds if and only if there is a $(j+1)$-relevant path connecting $v$ and $w$;
\item $\displaystyle \coeff_{j+1,\ell(v^{(j+1)})-\ell(w^{(j+1)})}(C_{w,v}^{(j)}(\lambda))\cdot x_{j+1}^{\ell(v^{(j+1)})-\ell(w^{(j+1)})}=
C_{w,v}^{(j+1)}(\lambda)$;
\item $C_{w,v}^{(j+1)}(\lambda)$, regarded as a polynomial in $x_i$, is homogeneous of degree $\ell( v^{(i)})-\ell(w^{(i)})$ for $1\leq i\leq j+1$.
\end{enumerate}
\end{lemma}
\begin{proof}
The same arguments as in the proof of Lemma \ref{relevant} show \textit{(i)} and \textit{(ii)}. Now \textit{(iii)} follows by induction on $j$ using \textit{(ii)}.
\end{proof}

For $0\leq j\leq n$ let $D_k^{(j)}(\lambda)$ be the hard Lefschetz determinant obtained from the $j^\text{th}$-degenerate action of $\lambda$, computed in the same bases used for $D_k(\lambda)$. 
We have
\begin{equation}\label{dkj}D_k^{(j)}(\lambda)=\sum_{\sigma}\sgn(\sigma)C^{(j)}_{w_1,\sigma(w_1)}(\lambda)C^{(j)}_{w_2,\sigma(w_2)}(\lambda)\ldots C^{(j)}_{w_{n(k)},\sigma(w_{n(k)})}(\lambda).
\end{equation}

For any $1\leq j\leq n$ let $\displaystyle 
M^{(j)}_k=\sum_{\ell(v)=\frac{d+k}2}\ell(v^{(j)})-\sum_{\ell(w)=\frac{d-k}2}\ell(w^{(j)})$.

\begin{lemma}\label{iter} For any $0\leq j\leq n-1$ we have:
\begin{enumerate}[i)]
\item  $\deg_{j+1} D_k^{(j)}(\lambda)\leq M_k^{(j+1)}$;
\item\label{22} $D_k^{(j)}(\lambda)$ is homogeneous of degree $M_k^{(i)}$ in $x_i$ for $1\leq i\leq j$;
\item\label{33} $\coeff_{j+1,M_k^{(j+1)}}D_k^{(j)}(\lambda)\cdot x_{j+1}^{M_k^{(j+1)}}=D_k^{(j+1)}(\lambda)$.
\end{enumerate}
\end{lemma}
\begin{proof}
Using \eqref{dkj} and Lemma \ref{Dkiter} this follows arguing just as in Lemma \ref{09}.
\end{proof}

Let $\displaystyle \mu_k=x_1^{M^{(1)}_k}x_2^{M^{(2)}_k}\cdot\ldots\cdot x_n^{M^{(n)}_k}$. We have the following:

\begin{cor}\label{corbrutto}
All monomials in  $D_k(\lambda)=D_k(x_1,\ldots,x_n)$ are smaller than $\mu_k$ in the lexicographic order.

The polynomial $D_k^{(n-1)}(\lambda)$ (which is equal to $D_k^{(n)}(\lambda)$) is homogeneous of degree $M_k^{(j)}$ in $x_j$ for any $1\leq j\leq n$, i.e. $D_k^{(n-1)}(\lambda)=R_k\mu_k$, with $R_k\in \kk$, and 
the coefficient of the monomial  $\mu_k$ in $D_k(\lambda)$ is $R_k$.
\end{cor}

\section{Hard Lefschetz for the maximal parabolic flag varieties}

To show that the polynomials $D_k(\lambda)$ are not identically zero, it suffices now to show that, for some ordering of the simple reflections, we have $R_k=(\mu_k)^{-1}D_k^{(n-1)}(\lambda)\in \kk^*$.
This will be done by investigating whether the $(n-1)^\text{th}$-degenerate action of a weight $\lambda$ has the HLP on $H^*(X,\kk)$. This coincides with the action
on $$H^*(G/P_{I_1},\kk)\otimes H^*(P_{I_1}/P_{I_2},\kk)\otimes\ldots \otimes H^*(P_{I_{n-1}}/B,\kk),$$
 where $\lambda=\sum x_i\varpi_i$ acts as multiplication by
 $$x_1\varpi_1\otimes 1\otimes\ldots\otimes 1+1\otimes x_2\varpi_2\otimes\ldots \otimes1+1\otimes 1\otimes\ldots\otimes x_n\varpi_n.$$

 \begin{ex}\label{5.1}
Let $W=\mathcal{S}_{n+1}$ be a Weyl group of type $A_n$. We label the simple reflections as follows:
\begin{center}
\begin{tikzpicture}
	\node (1) at (0,0) {$1$};
    \node(2) at (1.5,0) {$2$};
    \node (3) at (3,0) {$3$};
    \node (4) at (4.5,0) {$\cdots$};
    \node (5) at (6,0) {$(n-1)$};
    \node (6) at (7.5,0) {$n$} ;
    \path  (1) edge (2);
    \path  (2) edge (3);
    \path  (3) edge (4);
    \path  (4) edge (5);
    \path  (5) edge (6);
    \end{tikzpicture}
\end{center}
Then $P_{I_j}/P_{I_{j+1}}\cong \mathbb{P}^{n+1-j}(\C)$ for all $1\leq j \leq n$. So the degenerate action of $\lambda$ can be thought as multiplication by $\sum_{i=1}^n x_i\varpi_i$ on 
$\kk[\varpi_1,\ldots,\varpi_n]/(\varpi_1^{n+1},\ldots,\varpi_n^2)$.
\end{ex}

The aim of this section is to consider the action of the fundamental weight $\varpi_{j}$ on a single factor $H^*(P_{I_{j-1}}/P_{I_{j}},\kk)$. Notice that $\varpi_{j}$ has the HLP if and only if $x_j\varpi_j$ has the HLP for every (or any) $x_j\in \kk^*$.

We can assume $j=1$. Since we can choose arbitrarily the ordering $\{1,2,\ldots,n\}$ of $S$, for our goals it is enough for every irreducible root system to check the Hard Lefschetz Property on $H^*(G/P_{S\setminus\{1\}},\kk)$
for only one particular choice of  $\{1\}$. 

\begin{prop}\label{riassunto}
Let $\Phi$ be an irreducible root system with simple roots $S$. Then we can always choose $1\in S$ such that $\varpi_1$ has the HLP on $H^*(G/P_{S\setminus\{1\}},\kk)$ for any field of characteristic $p>|\Phi^+|$.
\end{prop}

\begin{proof}
We set $I=S\setminus \{1\}$. The proof is divides into cases.

\textbf{Case $A_n$:} We label the simple reflections as in Example \ref{5.1}.
We can choose $G=SL_{n+1}(\C)$. 
Then the parabolic flag variety $G/P_I$ is the Grassmannian of lines in $\C^{n+1}$, i.e. it is isomorphic to $\mathbb{P}^n(\C)$. Then $\varpi_1$ has the HLP in $H^*(G/P_I,\kk)\cong \kk[\varpi_1]/(\varpi_1^{n+1})$ for any field $\kk$. 

\textbf{Case $B_n$ and $C_n$:} We label the simple reflections as follows
\begin{center}
\begin{tikzpicture}
 	\node (1) at (0,0) {$1$};
    \node(2) at (1.5,0) {$2$};
    \node (3) at (3,0) {$3$};
    \node (4) at (4.5,0) {$\cdots$};
    \node (5) at (6,0) {$(n-1)$};
    \node (6) at (7.5,0) {$n$} ;
    \path  (1) edge (2);
    \path  (2) edge (3);
    \path  (3) edge (4);
    \path  (4) edge (5);
    \path (6.65,0.05) edge (7.3,0.05);
   	\path (6.65,-0.05) edge (7.3,-0.05);
 \end{tikzpicture}
\end{center}

If $W$ is the Weyl group of type $B_n$ (or $C_n$) is it easy to list all the elements in $W^I$ and to draw the parabolic Bruhat graphs $\mathfrak{B}_{B_n}^I$ and $\mathfrak{B}_{C_n}^I$  (here by $e$ we mean the identity of $W$).

Notice also that $P_I$ is cominuscule in type $B_n$ and minuscule in type $C_n$. The parabolic flag varieties $G/P_I$ are described in detail in these cases in \cite[\S 9.3]{BL}.

\begin{center}
\begin{tabular}{|l | l|}
\hline
\begin{tikzpicture}[scale=0.7]
	\foreach \x in {0,1,2,4.5,5.5,6.5,9}
{\fill (0,\x*1) circle (0.1);}
	\node at (-2.8,9) {$\mathfrak{B}_{B_n}^I$};
 	\node (1) at (0,0) [label=right:{\small $e$}] {};
    \node (2) at (0,1) [label=right:{\small $1$}] {};
    \node (3) at (0,2) [label=right:{\small $21$}] {};
    \node (4) at (0,3.25) {\small \vdots};
    \node (5) at (0,4.5) [label=right:{\small $ (n-1)\ldots21 $}] {};
    \node (6) at (0,5.5) [label=right:{\small $ n(n-1)\ldots21 $}] {};
    \node (7) at (0,6.5) [label=right:{\small $ (n-1)n(n-1)\ldots21 $}] {} ;
    \node (8) at (0,7.75) {\small\vdots};
    \node (9) at (0,9) [label=right:{\small$12\ldots(n-1)n(n-1)\ldots21$}] {};
    \path[->]  (1) edge node[left, scale=0.7]{$\alpha_1^\vee$}(2);
    \path[->]  (2)edge node[left, scale=0.7]{$\alpha_1^\vee$}(3);
    \path[->]  (3) edge node[left, scale=0.7]{$\alpha_1^\vee$}(4);
    \path[->]  (4) edge node[left, scale=0.7]{$\alpha_1^\vee$}(5);
	\path[->,color=blue]  (5) edge node[left, scale=0.7]{$2\alpha_1^\vee$}(6);
 	\path[->]  (6) edge node[left, scale=0.7]{$\alpha_1^\vee$}(7);
 	\path[->]  (7) edge node[left, scale=0.7]{$\alpha_1^\vee$}(8);
 	\path[->]  (8) edge node[left, scale=0.7]{$\alpha_1^\vee$}(9);

 \end{tikzpicture}

& 
\begin{tikzpicture}[scale=0.7]
	\foreach \x in {0,1,2,4.5,5.5,6.5,9}
{\fill (0,\x*1) circle (0.1);}
	\node at (-2.8,9) {$\mathfrak{B}_{C_n}^I$};
 	\node[scale=0.9] (1) at (0,0) [label=right:{\small $e$}] {};
    \node (2) at (0,1) [label=right:{\small $1$}] {};
    \node (3) at (0,2) [label=right:{\small $21$}] {};
    \node (4) at (0,3.25) {\small \vdots};
    \node (5) at (0,4.5) [label=right:{\small $(n-1)\ldots21 $}] {};
    \node (6) at (0,5.5) [label=right:{\small $n(n-1)\ldots21 $}] {};
    \node (7) at (0,6.5) [label=right:{\small $(n-1)n(n-1)\ldots21 $}] {} ;
    \node (8) at (0,7.75) {\small\vdots};
    \node (9) at (0,9) [label=right:{\small $12\ldots(n-1)n(n-1)\ldots21$}] {};
    \path[->]  (1) edge node[left, scale=0.7]{$\alpha_1^\vee$}(2);
    \path[->]  (2)edge node[left, scale=0.7]{$\alpha_1^\vee$}(3);
    \path[->]  (3) edge node[left, scale=0.7]{$\alpha_1^\vee$}(4);
    \path[->]  (4) edge node[left, scale=0.7]{$\alpha_1^\vee$}(5);
	\path[->]  (5) edge node[left, scale=0.7]{$\alpha_1^\vee$}(6);
 	\path[->]  (6) edge node[left, scale=0.7]{$\alpha_1^\vee$}(7);
 	\path[->]  (7) edge node[left, scale=0.7]{$\alpha_1^\vee$}(8);
 	\path[->]  (8) edge node[left, scale=0.7]{$\alpha_1^\vee$}(9);

 \end{tikzpicture}
\\ \hline\end{tabular}
\end{center}

From this it is evident that if $\Phi$ is of type $C_n$ then $\varpi_1$ has the HLP on $H^*(G/P_I, \kk)$ for every field $\kk$, while if $\Phi$ is of type $B_n$ then $\varpi_1$ has the HLP on $H^*(G/P_I,\kk)$ if and only if $\ch(\kk)\neq 2$. 

\textbf{Case $D_n$:} We label the simple reflections as follows:
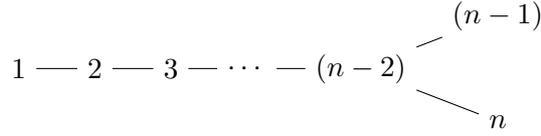
\begin{center}
\begin{tikzpicture}
 	\node (1) at (0,0) {$1$};
    \node(2) at (1,0) {$2$};
    \node (3) at (2,0) {$3$};
    \node (4) at (3,0) {$\cdots$};
    \node (5) at (4.5,0) {$(n-2)$};
    \node (6) at (6.3,0.7) {$(n-1)$} ;
     \node (7) at (6.3,-0.7) {$n$} ;
    \path  (1) edge (2);
    \path  (2) edge (3);
    \path  (3) edge (4);
    \path  (4) edge (5);
    \path  (5) edge (6);
    \path  (5) edge (7);
    \end{tikzpicture}
\end{center}

If $W$ is the Weyl group of type $D_n$ is it easy to list all the elements in $W^I$ and to draw parabolic Bruhat graph $\mathfrak{B}_{D_n}^I$.
Notice also that $P_I$ is minuscule and the parabolic flag variety $G/P_I$ is described in detail in \cite[\S 9.3]{BL}.

The parabolic Bruhat graph $\mathfrak{B}_{D_n}^I$ is:

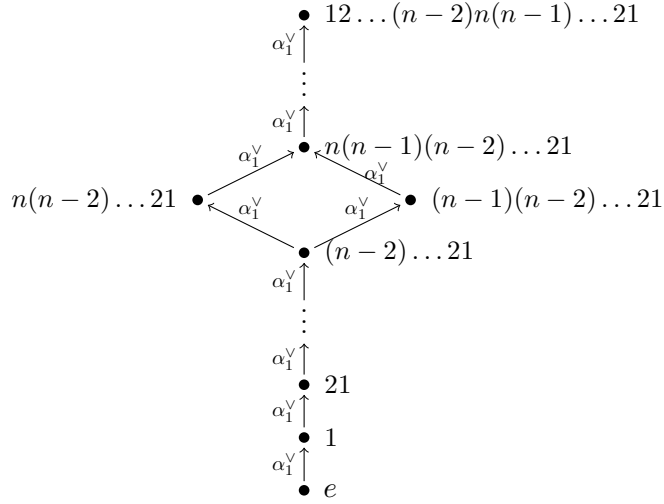
\vspace{-10pt}
\begin{center}
\begin{tikzpicture}[scale=0.7]

 	\node[scale=0.9] (1) at (0,0) [label=right:$e$] {};
    \node (2) at (0,1) [label=right:{\small $1$}] {};
    \node (3) at (0,2) [label=right:{\small $21$}] {};
    \node (4) at (0,3.35) {\small \vdots};
    \node (5) at (0,4.5) [label=right:{\small $(n-2)\ldots21 $}] {};
    \node (6) at (+2,5.5) [label=right:{\small $(n-1)(n-2)\ldots21 $}] {};
    \node (7) at (-2,5.5) [label=left:{\small $n(n-2)\ldots21 $}] {} ;
    \node (8) at (0,6.5) [label=right:{\small $n(n-1)(n-2)\ldots21 $}] {} ;
    \node (9) at (0,7.85) {\small\vdots};
    \node (10) at (0,9) [label=right:{\small $12\ldots(n-2)n(n-1)\ldots21$}] {};
    \path[->]  (1) edge node[left, scale=0.7]{$\alpha_1^\vee$}(2);
    \path[->]  (2)edge node[left, scale=0.7]{$\alpha_1^\vee$}(3);
    \path[->]  (3) edge node[left, scale=0.7]{$\alpha_1^\vee$}(4);
    \path[->]  (0,3.6) edge node[left, scale=0.7]{$\alpha_1^\vee$}(5);
	\path[->]  (5) edge node[above, scale=0.7]{$\alpha_1^\vee$}(7);
	\path[->]  (5) edge node[above, scale=0.7]{$\alpha_1^\vee$}(6);
 	\path[->]  (6) edge node[right, scale=0.7]{$\alpha_1^\vee$}(8);
 	\path[->]  (7) edge node[above, scale=0.7]{$\alpha_1^\vee$}(8);
 	\path[->]  (8) edge node[left, scale=0.7]{$\alpha_1^\vee$}(9);
 	\path[->]  (0,8.1) edge node[left, scale=0.7]{$\alpha_1^\vee$}(10);
 	\fill (6) circle (0.1);
 	\fill (7) circle (0.1);
 	\foreach \x in {1,2,3,5,8,10}
	{\fill (\x) circle (0.1);}
 \end{tikzpicture}
 \end{center}

It follows that $\varpi_1$ has the HLP if and only if $\ch(\kk)\neq 2$. 

\textbf{Exceptional Root Systems:}
We computed, with the help of the software Magma \cite{Magma}, for each of the exceptional Weyl groups the set of primes $p$ such that $\varpi_1$ has no HLP on $H^*(G/P_{S\setminus\{1\}},\kk)$ for an infinite field $\kk$ of characteristic $p$. 
We indicate in the Dynkin diagram the choice made for the vertex $1$.

\begin{center}
    \begin{tabular}{ | c | c | c | l |}
    \hline
    Root System & $|\Phi^+|$ & Dynkin Diagram & Primes with no HLP for $\varpi_1$ \\ \hline
	$F_4$ & $24$ & \begin{tikzpicture}[decoration={markings,mark=at position 0.7 with {\arrow{>}}}]
	\draw (0,0) -- (0.5,0);
	\draw (1,0) -- (1.5,0);
	\draw[double,postaction={decorate}] (0.5,0) -- (1,0);
	\foreach \x in {1,2,3}
    {\fill (\x*0.5,0) circle (0.04);}
    \fill (0,0) circle (0.04) node[above, scale=0.7] {$1$};
    \end{tikzpicture} & $2,3,13$ \\
    $G_2$ & $6$ &
\begin{tikzpicture}[decoration={markings,mark=at position 0.7 with {\arrow{>}}}]
	\draw[ double,postaction={decorate}] (0,0) -- (0.5,0);
	\draw[very thin] (0,0) -- (0.5,0);
    \fill (0.5,0) circle (0.04);
    \fill (0,0) circle (0.04) node[above, scale=0.7] {$1$};
    \end{tikzpicture}
    & $2$ \\
    $E_6$ & $36$ & \begin{tikzpicture}[decoration={markings,mark=at position 0.7 with {\arrow{>}}}]
\draw[thin] (0,0)--(1.4,0);
\draw[thin] (0.7,0)-- (0.7,0.35);
\foreach \x in {1,...,4}
{\fill (\x*0.35,0) circle (0.04);}
\fill (0.7,0.35) circle (0.04);
\fill (0,0) circle (0.04) node[above, scale=0.7] {$1$};
\end{tikzpicture} & $2,3,13$ \\
    $E_7$ & $63$ &
\begin{tikzpicture}[decoration={markings,mark=at position 0.7 with {\arrow{>}}}]
\draw[thin] (0,0)--(1.75,0);
\draw[thin] (1.05,0)-- (1.05,0.35);
\foreach \x in {1,...,5}
{\fill (\x*0.35,0) circle (0.04);}
\fill (1.05,0.35) circle (0.04);
\fill (0,0) circle (0.04) node[above, scale=0.7] {$1$};
\end{tikzpicture}    
    & $2,3,5,7,19,23$ \\
    $E_8$ & $120$ & \begin{tikzpicture}[decoration={markings,mark=at position 0.7 with {\arrow{>}}}]
\draw[thin] (0,0)--(2.10,0);
\draw[thin] (1.4,0)-- (1.4,0.35);
\foreach \x in {1,...,6}
{\fill (\x*0.35,0) circle (0.04);}
\fill (1.4,0.35) circle (0.04);
\fill (0,0) circle (0.04) node[above, scale=0.7] {$1$};
\end{tikzpicture}
& $ 2, 3, 5, 7, 19, 29, 31, 37, 41, 43, 47, 53$\\
    \hline
    \end{tabular}
\end{center}

This completes the proof of Proposition \ref{riassunto}.
\end{proof}

The following Lemma is standard:
\begin{lemma}\label{primitive}
Let $\kk$ be a field and $V=\bigoplus_{0\leq k\leq 2d} V^k$ a finite dimensional graded $\kk$-vector space. Let $\eta :V\ra V$ be a linear map of degree $2$ with the HLP, i.e. $\eta^k:V^{d-k}\ra V^{d+k}$ is an isomorphism for any $k$. Then there exists a decomposition of $V$, called the \emph{Lefschetz decomposition},  in the form
$$V=\bigoplus_{\stackrel{0\leq k\leq d}{1\leq i\leq r_k}} \kk[\eta]p_{k,i}$$ where $\{p_{k,i}\}_{1\leq i\leq r_k}$ is any basis of $V^{d-k}\cap \Ker(\eta^{k+1})$.  
\end{lemma}

In particular if $V=\bigoplus_{k,i} \kk[\eta]p_{k,i}$ is a primitive decomposition we get a basis $\{\eta^lp_{k,i}\}$  (with $0\leq k\leq d$, $1\leq i\leq r_k$ and $0\leq l\leq k$) of $V$.

The existence of the Lefschetz decomposition implies that, after changing the basis, the map $\eta$ can be represented by a graph which is a disjoint union of simple strings.

\begin{ex}
 Let $\Phi$ be of type $D_4$ with the reflections labeled as above. Then, if $\ch(\kk)\neq 2$, we can choose $\{P_e,P_1,P_{21},P_{321}+P_{421},P_{321}-P_{421},2P_{4321},2P_{24321},2P_{124321}\}$ as a basis of $H^*(G/P_{S\setminus\{1\}},\kk)$.
 In this basis multiplication by $\varpi_1$ is represented by the following graph:
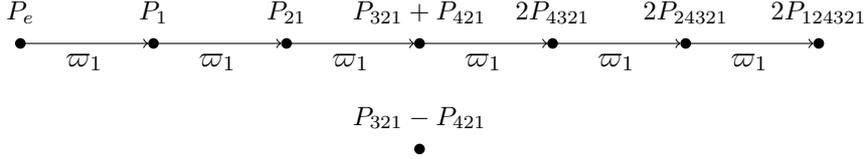
\begin{center}
\begin{tikzpicture}[scale=0.7]
  
	\foreach \x in {0,2.5,5,7.5,10,12.5}
		{\fill (\x,0) circle (0.1);
		 \path[->] (\x,0) edge node[below]{$\varpi_1$} (\x+2.4,0);}
		      \node at (0,0) [label=above:{\small $P_e$}] {};
		      \node at (2.5,0) [label=above:{\small $P_1$}] {};
		      \node at (5,0) [label=above:{\small $P_{21}$}] {};
		      \node at (7.5,0) [label=above:{\small $P_{321}+P_{421}$}] {};
		      \node at (10,0) [label=above:{\small $2P_{4321}$}] {};
		      \node at (12.5,0) [label=above:{\small $2P_{24321}$}] {};
		      \node at (15,0) [label=above:{\small $2P_{124321}$}] {};
		 \fill (15,0) circle (0.1);
	\fill (7.5,-2) circle (0.1);
	\node at (7.5,-2) [label=above:{\small $P_{321}-P_{421}$}] {};
\end{tikzpicture}
 \end{center}
\end{ex}

\section{Hard Lefschetz for Artinian Complete Intersection Monomial Rings}

In this section, let $\kk$ denote an arbitrary field of characteristic $p$.

\begin{teo}[\cite{Pr}]\label{Cook}
Let $A=\kk[\varpi_1,\varpi_2,\ldots,\varpi_n]/(\varpi_1^{d_1},\varpi_2^{d_2},\ldots,\varpi_n^{d_n})$. We regard $A$ as a graded algebra over $\kk$ in which the $\varpi_i$ have degree $2$. 
Let $d=\sum_{i=1}^n(d_i-1)$. Then if $p>d$ multiplication by $\lambda=\sum x_i\varpi_i$ has the HLP on $A$ if $x_i\in \kk^*$ for all $i$.
\end{teo}

Let $\lambda=\sum x_i\varpi_i$ with $x_i\in\kk^*$. In \cite[Corollary 2]{Pr}, Proctor actually gives a closed formula for the determinants $D_k(\lambda)$ of $\lambda^k:A^{d-k}\ra A^{d+k}$.
From Proctor's formula we can easily check that all the determinants are in $\kk^*$ if $p>d$, hence $\lambda$ has the HLP on $A$.

We give here an alternative elementary proof based on the representation theory of $\sl_2(\kk)$. Let
$$f=\begin{pmatrix}0 & 0\\ 1 & 0\end{pmatrix},\qquad h=\begin{pmatrix}1 & 0\\ 0 & -1\end{pmatrix},\qquad e=\begin{pmatrix}0 & 1\\ 0 & 0\end{pmatrix},$$
so that $\{f,h,e\}$ is a basis of $\sl_2(\kk)$.

For any integer $0\leq m\leq p-1$ let $L(m)$ be the irreducible $\sl_2(\kk)$-module of highest weight $m$. 
These modules can be obtained by reduction from the characteristic $0$ case, i.e. $L(m)$ has a basis $\{v_{m-2k}\}_{0\leq k\leq m}$ such that the action of $\sl_2(\kk)$ is described by
$$h\cdot v_i=iv_i,\qquad e\cdot v_i=\frac{m+i+2}{2}v_{i+2},\qquad f\cdot v_i=\frac{m-i+2}{2}v_{i-2}$$
for any $i$, where we set $v_{m+2}=v_{-m-2}=0$.

Let $U=\mathcal{U}(\sl_2(\kk))$ be the universal enveloping algebra of $\sl_2(\kk)$.
Let $M$ be a $\sl_2(\kk)$-module and let $v\in M$ be a highest weight vector of weight $a$ with $0\leq a\leq p-1$, i.e. $h\cdot v=av$ and $e\cdot v=0$.
Then $U\cdot v=\spa\langle f^k\cdot v\mid k\geq 0\rangle$ is a submodule of $M$ such that $\dim (U\cdot v) \geq a+1$. Moreover, $\dim (U\cdot v) = a+1$  if and only if $U\cdot v\cong L(a)$.

We consider the Casimir element $C=2ef+2fe+h^2\in U$. 
 It is easy to check that $C$ lies in the center of $U$, therefore $C$ acts as a scalar on any highest weight module $U\cdot v$. 
 If $v$ is of weight $m$, we get  $C\cdot v=(2ef+h^2)\cdot v=(2m+m^2)v$, so $C$ acts as the scalar $2m+m^2$ on $U\cdot v$.

\begin{prop}\label{CleGor}
 Let $m_1,m_2,\ldots,m_n$ be non-negative integers such that their sum  $d:=\sum_{i=1}^n m_i$ is smaller than $p$. Then $L(m_1)\otimes L(m_2)\otimes \ldots \otimes L(m_n)$ is a semisimple $\sl_2(\kk)$-module and
 it decomposes as $\bigoplus_{a=0}^{d} L(a)^{\nu_a}$, where $\nu_a$ are non-negative integers.
\end{prop}
\begin{proof}
By induction it is enough to consider the case $n=2$. Let $a=m_1$ and $b=m_2$. We can assume $a\geq b$. 
Let $\{v_{a-2k}\}_{0\leq k\leq a}$ (resp. $\{w_{b-2k}\}_{0\leq k\leq b}$) be a basis of $L(a)$ (resp. $L(b)$) as described above.

As in the characteristic $0$ case, for any integer $k$ with $0\leq k \leq b$, there exists a highest weight vector $v_{a-b+2k}\in L(a)\otimes L(b)$ of weight $a-b+2k$.
In fact, $e$ induces a map
$$e:\spa\langle v_i\otimes w_j\mid i+j=a-b+2k\rangle\lra\spa\langle v_i\otimes w_j\mid i+j=a-b+2k+2\rangle$$
which has a non-trivial kernel, as we can easily see by comparing the dimensions.

For any $k$, we have $U\cdot v_{a-b+2k}\cu \Ker(C-2(a-b+2k)-(a-b+2k)^2)$. Since $$2(a-b+2k)+(a-b+2k)^2\not\equiv 2(a-b+2h)-(a-b+2h)^2 \pmod p$$
for any $0\leq k,h\leq b$ with $k\neq h$ we get
$$\bigoplus_{k=0}^{b}U\cdot v_{a-b+2k}\cu \bigoplus_{k=0}^{b}\Ker(C-2(a-b+2k)-(a-b+2k)^2)\cu L(a)\otimes L(b).$$

Now, by comparing the dimensions we must have $U\cdot v_{a-b+2k}\cong L(a-b+2k)$, whence
$L(a)\otimes L(b)=L(a-b)\oplus L(a-b+2)\oplus \ldots \oplus L(a+b).$\end{proof}

\begin{proof}[Proof of Theorem \ref{Cook}]
For any $x\in \kk^*$, the algebra $\kk[\varpi]/(\varpi^a)$ can be seen as a $\sl_2(\kk)$-module, where $e$ acts as multiplication by $x\varpi$ and $h$ acts as multiplication by $2k-a+1$ on $\varpi^k$. 
If $a\leq p$, then $\kk[\varpi]/(\varpi^a)\cong L(a-1)$ as a $\sl_2(\kk)$-module.

Therefore, by Corollary \ref{CleGor}, if $d=\sum_{i=1}^n(d_i-1)<p$ the  algebra 
$$A\cong\kk[\varpi_1]/(\varpi_1^{d_1})\otimes \kk[\varpi_2]/(\varpi_2^{d_2})\otimes \ldots \otimes\kk[\varpi_n]/(\varpi_n^{d_n})$$
is semisimple as a $\sl_2(\kk)$-module, where $e$ acts as multiplication by $x_1\varpi_1+x_2\varpi_2+\ldots+x_n\varpi_n$ and $h$ acts as multiplication by $2\sum_{i=1}^n k_i -d$ on $\varpi_1^{k_1}\otimes \varpi_2^{k_2}\otimes \ldots \otimes \varpi_n^{k_n}$.
In particular $A$ is decomposed in a direct sum of $L(m)$, with $m\leq p-1$. Now hard Lefschetz easily follows since $e\in \sl_2(\kk)$ has the HLP on $L(m)$, for any $0\leq m \leq p-1$.
\end{proof}

\section{Proof of Theorem \ref{teo1}}

The case $\ch(\kk)\leq |\Phi^+|$ is discussed in Remark \ref{pleqPhi} and Remark \ref{rank2}, so we can assume $\ch(\kk)=p>|\Phi^+|$.

Let $\lambda=\sum_{i=1}^n x_i\varpi_i$ as before.
In view of Corollary \ref{corbrutto},  it remains to show that the polynomials $D_k^{(n-1)}(\lambda)$, with $1<k\leq n$, are non-zero for some indexing of $S=\{1,2,\ldots,n\}$. 
In other words we have to show that the $(n-1)^\text{th}$-degenerate action of $\lambda$ defined by the graph
$\BB^{(n-1)}:=\mathfrak{B}_{\Phi}^{I_1}\times\mathfrak{B}_{\Phi(I_1)}^{I_2}\times \ldots \times\mathfrak{B}_{\Phi(I_{n-1})}$ satisfies the hard Lefschetz theorem.

Since $\ch(\kk)=p>|\Phi^+|$, it follows from Proposition \ref{riassunto} that we can choose an ordering of $S$ such that,  for any $1\leq j\leq n$ and any $x_j\in \kk^*$, $x_j\varpi_j$ has the HLP on $H^*(P_{I_{j-1}}/P_{I_j},\kk)$.
Therefore, as in Lemma \ref{primitive}, we have a Lefschetz decomposition 
$$H^*(P_{I_{j-1}}/P_{I_j},\kk)=\bigoplus_{\stackrel{0\leq k \leq d_j}{1\leq i\leq r_k}}\kk[\varpi_j]p^j_{k,i}$$
where $\{p^j_{k,i}\}_{1\leq i\leq r_k}$ is a basis of $H^{d_j-k}(P_{I_{j-1}}/P_{I_j},\kk)\cap \Ker(\varpi_j^{k+1})$ and
$d_j=\dim(P_{I_{j-1}}/P_{I_j})$.

We obtain a decomposition
$$H^*(G/P_{I_1},\kk)\otimes H^*(P_{I_1}/P_{I_2},\kk)\otimes\ldots \otimes H^*(P_{I_{n-1}}/B,\kk)\cong$$
$$\cong\bigoplus_{\stackrel{i_1,i_2,\ldots,i_n}{k_1,k_2,\ldots,k_n}} \kk[\varpi_1]p^1_{k_1,i_1}\otimes \kk[\varpi_2]p^2_{k_2,i_2}\otimes\ldots\otimes\kk[\varpi_n]p^n_{k_n,i_n}\cong$$
$$\cong \bigoplus_{\stackrel{i_1,i_2,\ldots,i_n}{k_1,k_2,\ldots,k_n}} \kk[\varpi_1]/(\varpi_1^{k_1+1})\otimes \kk[\varpi_2]/(\varpi_2^{k_2+1})\otimes\ldots \otimes \kk[\varpi_n]/(\varpi_n^{k_n+1})\cong$$
$$\cong\bigoplus_{\stackrel{i_1,i_2,\ldots,i_n}{k_1,k_2,\ldots,k_n}} \kk[\varpi_1,\varpi_2,\ldots,\varpi_n]/(\varpi_1^{k_1+1},\varpi_2^{k_2+1},\ldots,\varpi_n^{k_n+1})$$
into $\lambda$-stable subspaces. 
Since $$\sum_{j=1}^n k_j\leq \sum_{j=1}^n d_j=\sum_{j=1}^n \dim \left( P_{I_{j-1}}/P_{I_j}\right)=\dim \left(G/B\right)=|\Phi^+|$$ from Theorem \ref{Cook} it follows that $\lambda$ has the HLP on every single summand of the decomposition.
Thus Theorem \ref{teo1} follows.

\bibliographystyle{alpha}

\Address

\end{document}